\documentclass[a4paper,10pt]{article}
\usepackage{amsmath,amssymb}
\usepackage{color}
\usepackage{tikz}

\newtheorem{theorem}{Theorem}[section] 
\newtheorem{proposition}[theorem]{Proposition} 
 
\newtheorem{corollary}[theorem]{Corollary} 
\newtheorem{lemma}[theorem]{Lemma}

\newtheorem{remark}[theorem]{Remark}

\newcommand{\TA}{\mathbb{A}}
\newcommand{\TD}{\mathbb{D}}
\newcommand{\TE}{\mathbb{E}}

\newcommand{\Fq}{\mathbb{F}_q}
\newcommand{\kk}{\mathbb{K}}
\newcommand{\ZZ}{\mathbb{Z}}
\newcommand{\QQ}{\mathbb{Q}}

\newcommand{\balpha}{\boldsymbol{\alpha}}
\newcommand{\bbeta}{\boldsymbol{\beta}}

\newcommand{\spec}{\operatorname{Spec}}

\newcommand{\lra}{\rightarrow}

\newenvironment{proof}{\begin{trivlist}\item{\bf{Proof.}}}
  {\hfill\rule{2mm}{2mm}\end{trivlist}}

\title{On the number of points over finite fields \\on varieties
  related to cluster algebras}
\author{F. Chapoton} \date{\today}

\begin{document}

\maketitle

\begin{abstract}
  We compute the number of points over finite fields of some algebraic varieties
  related to cluster algebras of finite type. More precisely, these
  varieties are the fibers of the projection map from the cluster
  variety to the affine space of coefficients.
\end{abstract}

\section{Introduction}

Cluster algebras have been introduced by S. Fomin and A. Zelevinsky
around 2000 \cite{cluster1,cluster2,cluster3,cluster4}, and have since
been a very active subject. Many connections have been found to
combinatorics \cite{fst}, Poisson geometry \cite{fg} and
representation theory \cite{bmrrt}. Cluster algebras are commutative
algebras, and can therefore be considered as objects of algebraic
geometry. In this article, we consider some algebraic varieties
closely related to the spectrum of cluster algebras. As a first step
towards computing their cohomology, we count their points over finite
fields.

More precisely, we use as a starting point a theorem \cite[Corollary
1.17]{cluster3} which gives a presentation by generators and relation
of acyclic cluster algebras. We use this presentation to define, for
each tree, a family of algebraic varieties depending on parameters in
an affine space.

In the classical cluster algebra setting, this affine space of
parameters corresponds to the so-called coefficients variables in
cluster algebras. The fiber over a point in our affine space of
parameters is a fiber of the map from the spectrum of the cluster
algebra to the spectrum of its coefficient ring.

For simplicity, we restrict ourself to simply-laced cluster algebras
of finite type, which are indexed by the usual $\TA$-$\TD$-$\TE$ list
of Dynkin diagrams. We have implicitly chosen to work with the
alternating orientation of these Dynkin diagrams, but the results do
not depend on the orientation.

We are also able to obtain more precise results in the case of type
$\TA$. Here, we describe the cohomology with compact support, in the
even case.

The interest of the results may be in the very simple shape of the
answers in the generic case (formulas \eqref{genAeven},
\eqref{genAodd} in type $\TA$ and formulas \eqref{genDodd},
\eqref{genDeven} in type $\TD$), which are nice polynomials in the
cardinal $q$ of the finite field $\Fq$. It would be very interesting
to see if the cohomology with compact support is as simple as one may
expect from these nice polynomials. We prove that this is indeed the
case in type $\TA_n$ with $n$ even.

\section{General results}

In this section, we introduce general definitions and tools valid for all trees. Later on, we will use these for finite-type simply-laced Dynkin diagrams.

\subsection{Definition}

Let $T$ be a tree, \textit{i.e.} a finite graph which is connected and
simply connected. We will write $s-t$ to denote that $s$ and $t$ are
adjacent vertices of $T$.

Let $\balpha=(\alpha_t)_{t\in T}$ be a function on the set of vertices
of $T$ with values in some field $\kk$.

Let us call $X_T(\balpha)$ the affine scheme over $\kk$
defined by
\begin{equation}
 \label{exchange}
  x_t x'_t= 1 +\alpha_t \prod_{s-t} x_s,
\end{equation}
for all vertices $t$ of $T$.

One can also consider this set of equations as defining a family of
schemes over the base affine scheme $\spec \ZZ[(\alpha_t)_{t\in
  T}]$. We will study the fibers of this family. From now on, we will
assume (unless explicitly stated otherwise) that the $\alpha_t$ are
invertible. This amounts to restrict the family to $\spec
\ZZ[(\alpha_t,\alpha_t^{-1})_{t\in T}]$.

\begin{remark}
  Instead of a tree $T$, one can also consider a disjoint union of
  trees $F$, in which case the variety $X_F(\balpha)$ will be the
  product of the varieties associated with the connected components of
  $F$.
\end{remark}

\subsection{Reduction using domino tiling}

Every tree $T$ is a bipartite graph. Let us fix a choice of black
and white vertices such that every edge has a white end and a black
end.

\begin{lemma}
  \label{saute_mouton}
  Let $s$ and $t$ be adjacent vertices. Let $\bbeta$
  be the function defined by
  \begin{align*}
    \beta_s&=1,\\
    \beta_u&=\alpha_u/\alpha_s&\text{ if }u\not= s\text{ and }u-t,\\
    \beta_u&=\alpha_u&\text{else.}
  \end{align*}
  Then $X_T(\balpha)$ is isomorphic to $X_T(\bbeta)$.
\end{lemma}
\begin{proof}
  One uses the change of variable $x_t=\alpha_s x_t$ and
  $x'_t=x'_t/\alpha_s$.
\end{proof}

Note that this operation only changes the values of the function on
vertices that have the same color as $s$. In a pictorial way, the
value at vertex $s$ jumps over vertex $t$ and get spread over the
other neighbors of $t$.


A \textbf{partial domino tiling} of $T$ is a subset of the set of
edges such that every vertex appears at most once among the ends of
the chosen edges. It is called \textbf{full} if every vertex appears
exactly once.

\begin{lemma}
  \label{white_leaf}
  Let $T$ be a tree with a full domino tiling. Then $T$ has a white leaf.
\end{lemma}
\begin{proof}
  By an easy induction on the number of vertices.
\end{proof}

\begin{proposition}
  \label{reduction}
  Let $T$ be a tree, endowed with a partial domino tiling. Every
  $X_T(\balpha)$ is isomorphic to some $X_T(\bbeta)$ where
  $\beta_s=1$ for every vertex $s$ which is covered by a domino.
\end{proposition}
\begin{proof}
  The proof uses only Lemma \ref{saute_mouton} to modify the function $\balpha$. One can therefore treat black and white vertices separately. Let us just prove the statement for white vertices, the case of black
  vertices being the same with colors exchanged.

  Let $C$ be the set of vertices covered by dominoes. The partial
  domino tiling of $T$ defines for every white vertex $s\in C$ a canonical
  black neighbor $B(s)$.

  Let us orient the edges of $T$ from white to black inside the dominoes (from $s$ to $B(s)$) and from black to white outside the dominoes.

  Then the induced subgraph of $T$ on the set $C$ of vertices covered
  by dominoes is a disjoint union of trees with oriented edges. One
  has to flip every white vertex $s\in C$ over the black vertex $B(s)$
  (which means using Lemma \ref{saute_mouton} to replace $\balpha$ by
  a modified function), in a well-chosen order. This order must be
  compatible with the partial order given by the orientation of edges
  : one has to start with white leaves (which exist by Lemma
  \ref{white_leaf}).

  At the end of the process, one obtains a function $\bbeta$ such that every white vertex in $C$ has value $1$ and such that $X_T(\balpha)$ is isomorphic to $X_T(\bbeta)$.
\end{proof}


\begin{remark}
  For every tree, one can find a partial domino tiling where the
only vertices which are not covered are leaves.
\end{remark}

\subsection{Induction by removal of leaves}

Let $T$ be a tree and let $f$ be a leaf of $T$. Let $g$ be the unique vertex
adjacent to $f$. Let $\balpha$ be a function on $T$.

Let $T'$ be the tree obtained from $T$ by removing $f$. For every
element $\beta$ in the ground field $\kk$, let $\balpha'(\beta)$ be the
function on $T'$ defined by
\begin{align}
  \alpha'_g(\beta)&=\alpha_g \beta,\\
  \alpha'_s(\beta)&=\alpha_s \text{ if }s\not=g.
\end{align}

Let $T''$ be the tree or disjoint union of trees obtained from $T$ by
removing $f$ and $g$. Let $\balpha''$ be the function on $T''$ defined
by
\begin{align}
  \alpha''_s&=-\alpha_s/\alpha_f &\text{ if }s-g \text{ in }T,\\
  \alpha''_s&=\alpha_s &\text{ else.}
\end{align}

\begin{proposition}
  \label{leaf_removal}
  The scheme $X_T(\balpha)$ is the disjoint union of the scheme $A_1
  \times X_{T''}(\balpha'') $ and of a variety fibered over $A_1
  \setminus \{0\}$ with fiber $X_{T'}(\balpha'(\beta))$ over $\beta$.
\end{proposition}
\begin{proof}
  One simply has to separate points according to whether $x_f=0$
  (which gives the first part) or not (which gives the variety fibered over
  $A_1 \setminus \{0\}$).

  If one assumes $x_f=0$, the equation \eqref{exchange} for vertex $f$
  gives the invertible value $-1/\alpha_f$ for $x_g$. Then the equation
  \eqref{exchange} for vertex $g$ gives a value for $x'_g$. There
  remains a free variable $x'_f$ and the equations for $T''$ with
  function $\balpha''$. This corresponds to the product of the affine
  space $A_1$ and the variety $X_{T''}(\balpha'')$
 
  If one now assumes that $x_f$ is invertible, then the equation
  \eqref{exchange} for vertex $f$ gives a value to $x'_f$. One can
  remove this equation; there remains the equations for $T'$ with a
  function depending on the value of $x_f$. If one moreover fixes the
  value of $x_f$ to be an invertible element $\beta$ of $\kk$, then
  one gets the equations of $X_{T'}(\balpha'(\beta))$.
\end{proof}

\section{Type $\TA$}

We will now consider the Dynkin diagrams of type $\TA$.

\begin{center}
\begin{tikzpicture}[scale=0.6]
\tikzstyle{every node}=[draw,shape=circle,very thick,fill=white]

\draw (0,0) node[fill=blue!50] {} -- (1,0) node {} -- (2,0) node {} -- (3,0) node {} -- (4,0) node {} -- (5,0) node{};

\end{tikzpicture}
\end{center}

\subsection{Number of points over finite fields}

In type $\TA_n$, the cluster algebra is generated by $n$ cluster
variables $x_1,\dots,x_n$ (which form a cluster), the $n$ adjacent
cluster variables $x'_1,\dots,x'_n$ and $n$ coefficient variables
$\alpha_1,\dots,\alpha_n$ with the following relations:
\begin{align*}
  x_1 x'_1 & = 1 + \alpha_1 x_2,\\
  x_2 x'_2 & = 1 + \alpha_2 x_1 x_3,\\
  & \dots \\
  x_{n-1} x'_{n-1} & = 1 + \alpha_{n-1} x_{n-2} x_n,\\
  x_n x'_n & = 1 + \alpha_n x_{n-1}.
\end{align*}

Let us call $X_{\TA_n}(\alpha_1,\dots,\alpha_n)$ this variety. As
before, we assume that the $\alpha_i$ are invertible.

\begin{proposition}
  If $n$ is even, then $X_{\TA_n}(\alpha_1,\dots,\alpha_n)\simeq X_{\TA_n}(1,1,\dots,1)$. 

  If $n$ is odd, then $X_{\TA_n}(\alpha_1,\dots,\alpha_n)\simeq X_{\TA_n}(\alpha,1,\dots,1)$, for some $\alpha$ depending only on the $\alpha_i$ with $i$ odd.
\end{proposition}
\begin{proof}
  This is obtained by applying Proposition \ref{reduction} to the
  obvious full domino tiling (even case) or to the partial domino
  tiling avoiding only the first vertex (odd case).
\end{proof}

For short, we will denote $X_{n}(\alpha)$ for
$X_{\TA_n}(\alpha,1,\dots,1)$. For the number of points of these
varieties over finite fields, we will use the following notation:
$N_{\TA_n}(\alpha)$ is the number of points of $X_n(\alpha)$. When $n$
is even, we will also use $N_{\TA_n}$ for short.

\begin{proposition}
  \label{valeurA}
  If $n$ is even, then
  \begin{equation}
    \label{genAeven}
    N_{\TA_n}=\frac{q^{n+2}-1}{q^2-1}.
  \end{equation}
  If $n$ is odd and $\alpha \not = (-1)^{(n+1)/2}$, then
   \begin{equation}
     \label{genAodd}
     N_{\TA_n}(\alpha)=\frac{(q^{(n+1)/2}-1)(q^{(n+3)/2}-1)}{q^2-1}.
  \end{equation}
  If $n$ is odd, then
  \begin{equation}
        N_{\TA_n}((-1)^{(n+1)/2})=\frac{(q^{(n+1)/2}-1)(q^{(n+3)/2}-1)}{q^2-1}+q^{(n+1)/2}.
  \end{equation}
\end{proposition}

\begin{proof}
  By induction using leaf-removal (Proposition
  \ref{leaf_removal}). The statement is clear if $n=0$, in which case
  the variety $X_{\TA_0}(\,)$ is just a point. It is also immediate if
  $n=1$.

  Let us first note that the statement implies (for $n$ even or odd) that
  \begin{equation}
    \label{sum_aux}
    \sum_{\alpha \in \Fq^*} N_{\TA_n}(\alpha)=\frac{q^{n+2}+(-1)^{n+1}}{q+1}.
  \end{equation}

  Assume that $n$ is even. Then Proposition \ref{leaf_removal} becomes
  \begin{equation}
    N_{\TA_n}= q N_{\TA_{n-2}}+\sum_{\alpha \in \Fq^*} N_{\TA_{n-1}}(\alpha),
  \end{equation}
  which can be rewritten (using \eqref{sum_aux}) as
  \begin{equation}
    N_{\TA_n}= q N_{\TA_{n-2}}+\frac{q^{n+1}+1}{q+1}.    
  \end{equation}
  This implies the expected formula for $N_{\TA_n}$.

  Assume that $n$ is odd.  Then Proposition \ref{leaf_removal} becomes
  \begin{equation}
    N_{\TA_n}(\alpha)= q N_{\TA_{n-2}}(-1/\alpha)+(q-1) N_{\TA_{n-1}},
  \end{equation}
  which can be rewritten as
  \begin{equation}
    N_{\TA_n}(\alpha)= q N_{\TA_{n-2}}(-1/\alpha)+\frac{q^{n+1}-1}{q+1}.
  \end{equation}

  Note that $\alpha = (-1)^{(n+1)/2}$ if and only if
  $-1/\alpha =(-1)^{(n-2+1)/2}$. The induction hypothesis then
  implies the expected formulas for $N_{\TA_n}(\alpha)$.
\end{proof}

Let $Y_{\TA_n}$ be the union of all varieties $X_{n}(\alpha)$ for
$\alpha$ invertible. By a natural convention, $Y_{\TA_0}$ is just $A^1
\setminus \{0\}$. Let us note as a lemma the formula \eqref{sum_aux}
that we have obtained in the proof of Prop. \ref{valeurA}.
\begin{lemma}
  \label{valeurAsomme}
  For every $n\geq 0$, one has
  \begin{equation}
    \sum_{\alpha \in \Fq^*} N_{\TA_n}(\alpha)=\frac{q^{n+2}+(-1)^{n+1}}{q+1}.
  \end{equation}
  This is the number of points on $Y_{\TA_n}$ over the finite field
  $\Fq$.
\end{lemma}

Let now $Z_{\TA_n}$ be the union of all varieties
$X_{n}(\alpha)$ for any $\alpha$ (we do not assume here
that $\alpha$ is invertible).
\begin{proposition}
  \label{Yopen_Yclosed}
  For every $n\geq 1$, the space $Z_{\TA_n}$ is the disjoint union of $Y_{\TA_n}$
  and $Y_{\TA_{n-1}}$. The number of points on $Z_{\TA_n}$ over the finite field
  $\Fq$ is $q^{n+1}$.
\end{proposition}
\begin{proof}
  The proof of this decomposition is obvious: either $\alpha=0$ (and
  one obtains $Y_{\TA_{n-1}}$) or not (in which case one gets
  $Y_{\TA_n}$). The counting result follows from Lemma \ref{valeurAsomme}.
\end{proof}

This result immediately suggests that $Z_{\TA_n}$ may just be an affine
space. This is proved in the next section and will allow to compute
some cohomology groups.

We will use later the following result.
\begin{lemma}
  \label{Y_product}
  For every even $n\geq 1$, the space $Y_{\TA_n}$ is isomorphic to the
  product of $X_{n}(1)$ with $A^1 \setminus \{0\}$.
\end{lemma}
\begin{proof}
  This is a simple consequence of Lemma \ref{saute_mouton}. One has
  clearly a fibration, and this is made trivial by a simple change of
  variables.
\end{proof}

\subsection{Cohomology with compact supports}

\begin{proposition}
  \label{z_is_a}
  For every $n\geq 1$, there is a surjective morphism $\phi$ from
  $Z_{\TA_{n+1}}$ to $Z_{\TA_n}$ with fiber $A^1$. For every $n\geq
  1$, there is an isomorphism $Z_{\TA_n} \simeq A^{n+1}$.
\end{proposition}
\begin{proof}
  The morphism $\phi$ from $Z_{\TA_{n+1}}$ to $Z_{\TA_n}$ is defined
  by forgetting the first equation \eqref{exchange}:
  \begin{equation}
    x_1 x'_1 = 1+\alpha x_2.
  \end{equation}
  One simply has to shift down the indices of variables $x_i$ and
  $x'_i$ for $i\geq 2$ and let $x_1$ play the role of $\alpha$.

  As it is not possible that both $x_1$ and $x_2$ vanish (by the
  second equation \eqref{exchange}), the first equation is the
  equation of a line in the plane with coordinates
  $x'_1,\alpha$. Therefore every fiber of $\phi$ is a line.

  One can easily check that $Z_{\TA_1} \simeq A^{2}$. Then the
  expected isomorphism follows by induction.
\end{proof}

\begin{corollary}
  As an open set of $Z_{\TA_n}$, $Y_{\TA_n}$ is smooth.
\end{corollary}

Recall that the cohomology with compact support of the affine space $A^{n}$ is very
simple: the only non-zero group is $H_c^{2n}(A^{n})\simeq \QQ(n)$,
where $\QQ(n)$ is the Tate Hodge structure of weight $n$.

\begin{proposition}
  \label{coho_Y}
  For $n\geq 0$, the non-zero cohomology groups with compact support
  of $Y_{\TA_n}$ are
  \begin{equation}
    H_c^{i+n+1}(Y_{\TA_n})\simeq \QQ(i),
  \end{equation} 
 for $ 0 \leq i \leq n+1$.
\end{proposition}
\begin{proof}
  The proof is by induction on $n$. The statement is true if
  $n=0$. One then uses the long exact sequence in cohomology with compact
  support for the open-closed decomposition $Z_{\TA_n}=Y_{\TA_n}\sqcup
  Y_{\TA_{n-1}}$ (see Prop. \ref{Yopen_Yclosed}), together with Prop. \ref{z_is_a}.
\end{proof}

One can then obtain the cohomology with compact support of $X_{n}(1)$
when $n$ is even.
\begin{proposition}
  For $n\geq 0$ even, the non-zero cohomology groups with compact
  support of $X_{n}(1)$ are
  \begin{equation}
    H_c^{i+n}(X_{n}(1))\simeq \QQ(i),
  \end{equation}
  for all even $i$ between $0$ and $n$.
\end{proposition}
\begin{proof}
  By induction on $n$. The statement is true for $n=0$ with the
  natural convention that $X_{\TA_0}(\,)$ is a point.

  One uses two ingredients. The first one is the long exact sequence
  for the open-closed decomposition $X_{n}(1)=Y_{\TA_{n-1}} \sqcup A^1
  \times X_{n-2}$. The second one is the K{\"u}nneth isomorphism
  describing the cohomology of the product $Y_{\TA_n} \simeq
  A^1\setminus \{0\} \times X_{n}(1)$ (see Lemma \ref{Y_product}). We
  also need the fact that the cohomology with compact support of
  $A^1\setminus \{0\}$ is $\QQ(0)$ in degree $1$ and $\QQ(1)$ in
  degree $2$.

  From the long exact sequence, one gets exact sequences
  \begin{equation}
    0 \lra \QQ(i) \lra H_c^{i+n}(X_{n}(1)) \lra \QQ(i+1) \lra \QQ(i+1) \lra H_c^{i+n+1}(X_{n}(1)) \lra 0,
  \end{equation}
  for even $i$ between $0$ and $n$.

  One would like to conclude that $H_c^{i+n}(X_{\TA_{n+2}}) \simeq
  \QQ(i)$ and $H_c^{i+n+1}(X_{\TA_{n+2}}) \simeq 0$.

  Assume on the contrary that, for some even $i$,
  $H_c^{i+n+1}(X_{\TA_{n}}) \simeq \QQ(i+1)$ and
  $H_c^{i+n}(X_{\TA_{n}})$ is an extension of $\QQ(i)$ by $\QQ(i+1)$,
  hence has dimension $2$.

  Then the K{\"u}nneth formula would imply that
  $H_c^{i+n+1}(Y_{\TA_{n}})$ has dimension at least $2$, which is
  absurd, as $H_c^{i+n+1}(Y_{\TA_{n}})$ is $\QQ(i)$ by
  Prop. \ref{coho_Y}.
\end{proof}

\subsection{Smoothness}

Let us prove that the varieties $X_n(\alpha)$ are smooth for generic
$\alpha$. Recall that $\alpha$ is assumed to be invertible.

\begin{proposition}
  For $n$ even, $X_n(\alpha)$ is smooth.

  For $n$ odd and $\alpha \not = (-1)^{(n+1)/2}$, $X_n(\alpha)$ is smooth.

  For $n$ odd and $\alpha = (-1)^{(n+1)/2}$, $X_n(\alpha)$ has a
  unique singular point: $x_i=x'_i=0$ for odd $i$ and $x_i=x'_i=-(-1)^{(n+i)/2}$
  for even $i$.
\end{proposition}
\begin{proof}
  The proof is by induction on $n$. The statement is clear if $n=0,1$.

  Assume that there is a singular point on $X_{n+2}(\alpha)$.
  
  The equations defining a singular point on $X_{n+2}(\alpha)$ are the
  $n$ equations of $X_{n+2}(\alpha)$ together with the vanishing of
  all minors of rank $n$ of the $2n \times n$ matrix $M_{n+2}(\alpha)$
  of partial derivatives of these $n$ equations with respect to
  variables $x_1,\dots,x_n,x'_1,\dots,x'_n$. This matrix
  $M_{n+2}(\alpha)$ looks as follows:
  \begin{equation}
    \begin{bmatrix}
      x'_1 & -\alpha & 0 & \dots & 0            &x_1 & 0 & & \dots & 0 \\
      -x_3 & x'_2 & -x_1 & 0  &          &0 & x_2 & 0 & &  \\
      0 & -x_4 & x'_3 & -x_2 & 0        &0 & 0   & x_3 & 0 &  \\
      &\ddots&\ddots&\ddots&                        &&&\ddots&\ddots&\ddots\\
      0 & \dots & 0 & -1 & x'_n        &0 & \dots & & 0 & x_n
    \end{bmatrix}
  \end{equation}

  Let us distinguish $2$ cases and some sub-cases.

  \textbf{First case} : $x_1=0$.
  
  Using the equations, this hypothesis implies that
  $x_2=-1/\alpha$ and $x'_2=-\alpha$.

  \underline{Assume} first that $x'_1=0$. Then the vanishing of all
  minors of the matrix $M_{n+2}(\alpha)$ reduces to the vanishing of
  all minors of the matrix $M_{n}(-1/\alpha)$ (with a shift of indices
  by $2$). So the singular point gives, by restriction to coordinates
  $(x_i,x'_i)_{i\geq 3}$, a singular point on $X_{n}(-1/\alpha)$.

  If $n+2$ is even or $n+2$ is odd and $\alpha\not=(-1)^{(n+2+1)/2}$, this
  is absurd by induction hypothesis.

  If $n+2$ is odd and $\alpha=(-1)^{(n+2+1)/2}$, there is only one
  solution by induction hypothesis: $x_i=x'_i=0$ for odd $i\geq 3$ and
  $x_i=x'_i=-(-1)^{(n+i)/2}$ for even $i\geq 2$. It is readily checked that the
  point $(x_i,x'_i)_{i\geq 1}$ is indeed a singular point on $X_{n+2}(\alpha)$.

  \underline{Assume} on the contrary that $x'_1\not=0$. Then the
  vanishing of all minors of the matrix $M_{n+2}(\alpha)$ reduces to
  the vanishing of all minors of an extended matrix which is made of
  $M_n(-1/\alpha)$ plus one more column on the left:
  \begin{equation}
    \begin{bmatrix}
      -x_4   & x'_3 & -x_2 & 0    & \dots & 0  & x_3 & 0   &      \dots & 0 \\
      0      & -x_5 & x'_4 & -x_3 & 0     &    & 0   & x_4 & 0          &   \\
      \vdots &     &\ddots &\ddots&\ddots&    &    &  &\ddots  &\ddots\\
      0      & 0    &\dots & 0   & -1    &x'_n& 0    &\dots  & 0     &  x_n
    \end{bmatrix}
  \end{equation}
  In particular, by restriction to coordinates $(x_i,x'_i)_{i\geq 3}$,
  we obtain a singular point on $X_{n}(-1/\alpha)$.

  If $n+2$ is even or $n+2$ is odd and $\alpha\not=(-1)^{(n+2+1)/2}$, this
  is absurd by induction hypothesis.
  
  If $n+2$ is odd and $\alpha=(-1)^{(n+2+1)/2}$, there is only one
  solution by induction hypothesis: $x_i=x'_i=0$ for odd $i\geq 3$ and
  $x_i=x'_i=-(-1)^{(n+i)/2}$ for even $i\geq 2$. One can then check that the
  extended matrix has a non-vanishing minor at this point, and
  therefore $(x_i,x'_i)_{i\geq 1}$ is not a singular point on
  $X_{n+2}(\alpha)$. This is absurd.

  \textbf{Second case} : $x_1\not=0$. Then the vanishing of all minors
  of the matrix $M_{n+2}(\alpha)$ reduces to the vanishing of all minors of an
  an extended matrix which is made of $M_{n+1}(x_1)$ (with a shift of
  indices by $1$) plus one more column on the left:
  \begin{equation}
    \begin{bmatrix}
      -x_3   & x'_2 & -x_1 & 0    & \dots & 0  & x_2 & 0   &      \dots & 0 \\
      0      & -x_4 & x'_3 & -x_2 & 0     &    & 0   & x_3 & 0          &   \\
      \vdots &     &\ddots &\ddots&\ddots&    &    &  &\ddots  &\ddots\\
      0      & 0    &\dots & 0   & -1    &x'_n& 0    &\dots  & 0     &  x_n
    \end{bmatrix}
  \end{equation}
  In particular, by restriction to coordinates $(x_i,x'_i)_{i\geq 2}$,
  we obtain a singular point on $X_{n+1}(x_1)$.

  If $n+2$ is odd, this is absurd by induction hypothesis.

  If $n+2$ is even, then there is only one possible solution by induction
  hypothesis: $x_i=x'_i=0$ for even $i\geq 2$ and $x_i=x'_i=-(-1)^{(n+1+i)/2}$ for
  odd $i\geq 3$. One can then check that the extended matrix has a
  non-vanishing minor at this point, which is therefore not a singular
  point on $X_{n+2}(\alpha)$.
 
\end{proof}







\section{Type $\TD$}

We will now consider the Dynkin diagrams of type $\TD$.

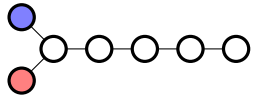
\begin{center}
\begin{tikzpicture}[scale=0.6]
\tikzstyle{every node}=[draw,shape=circle,very thick,fill=white]

\draw (0.3,0.7) node[fill=blue!50] {} -- (1,0) -- (2,0) node {} -- (3,0) node {} -- (4,0) node {} -- (5,0) node{};

\draw (0.3,-0.7) node[fill=red!50] {} -- (1,0) node {};

\end{tikzpicture}
\end{center}

In type $\TD_n$, the cluster algebra is generated by $n$ cluster
variables $x_1,x_2,x_3,\dots,x_n$ (which form a cluster), the $n$ adjacent
cluster variables $x_1',x_2',x'_3,\dots,x'_n$ and $n$ coefficient variables
$\alpha_1,\dots,\alpha_n$ with the following relations:
\begin{align*}
  x_1 x'_1 & = 1 + \alpha_1 x_3,\\
  x_2 x'_2 & = 1 + \alpha_2 x_3,\\
  x_3 x'_3 & = 1 + \alpha_3 x_1 x_2 x_4,\\
  x_4 x'_4 & = 1 + \alpha_4 x_3 x_5,\\
  & \dots \\
  x_{n-1} x'_{n-1} & = 1 + \alpha_{n-1} x_{n-2} x_n,\\
  x_n x'_n & = 1 + \alpha_n x_{n-1}.
\end{align*}

Let us call $X_{\TD_n}(\alpha_1,\dots,\alpha_n)$ this variety. As
before, we assume that the $\alpha_i$ are invertible.

\begin{proposition}
  If $n$ is even, then $X_{\TD_n}(\alpha_1,\dots,\alpha_n)\simeq
  X_{\TD_n}(\alpha,\beta,1,\dots,1)$, for some $\alpha,\beta$ depending
  on the $\alpha_i$.

  If $n$ is odd, then $X_{\TD_n}(\alpha_1,\dots,\alpha_n)\simeq X_{\TD_n}(\alpha,1,\dots,1)$, for some $\alpha$ depending on the $\alpha_i$.
\end{proposition}
\begin{proof}
  This is obtained by applying Proposition \ref{reduction} to the
  partial domino tiling avoiding the first two vertices (even case) or
  to the partial domino tiling avoiding only the first vertex (odd
  case).
\end{proof}

Let us introduce some notation for the number of points of these
varieties over finite fields. If $n$ is odd, we will denote
$N_{\TD_n}(\alpha)$ the number of points of
$X_{\TD_n}(\alpha,1,\dots,1)$. If $n$ is even, we will denote
$N_{\TD_n}(\alpha,\beta)$ the number of points of
$X_{\TD_n}(\alpha,\beta,1,\dots,1)$.

\begin{proposition}
  \label{valeurD}
  If $n$ is odd and $\alpha \not= 1$, then
  \begin{equation}
    \label{genDodd}
    N_{\TD_n}(\alpha)=q^n-1.
  \end{equation}
  If $n$ is odd, then  
  \begin{equation}
    N_{\TD_n}(1)=q^n-1+q^2 \frac{q^{n-1}-1}{q^2-1}.
  \end{equation}
  If $n$ is even, $\alpha\not=\beta$, $\alpha\not=(-1)^{n/2}$ and
  $\beta\not=(-1)^{n/2}$, then
  \begin{equation}
    \label{genDeven}
    N_{\TD_n}(\alpha,\beta)=(q^{n/2}-1)^2.
  \end{equation}
  If $n$ is even, and $\alpha=\beta$ differs from $(-1)^{n/2}$, then
  \begin{equation}
     N_{\TD_n}(\alpha,\alpha)=(q^{n/2}-1)^2+q^2\frac{(q^{(n-2)/2}-1)(q^{n/2}-1)}{q^2-1}.
  \end{equation}
  If $n$ is even, and $\alpha\not=\beta$ and $\alpha=(-1)^{n/2}$, then
  \begin{equation}
     N_{\TD_n}((-1)^{n/2},\beta)=(q^{n/2}-1)^2+(q-1) q^{n/2}.
  \end{equation}
  If $n$ is even, and $\alpha=\beta=(-1)^{n/2}$, then $ N_{\TD_n}((-1)^{n/2},(-1)^{n/2}) $ equals
  \begin{equation}
    (q^{n/2}-1)^2+2(q-1) q^{n/2}+q^2\frac{(q^{(n-2)/2}-1)(q^{n/2}-1)}{q^2-1}+q^{(n+2)/2}.
  \end{equation}
\end{proposition}

\begin{proof}
  The proof uses leaf-removal (Proposition \ref{leaf_removal}) and
  knowledge of type $\TA$.

  If $n=3$, one has $X_{\TD_3}(\balpha)\simeq X_{\TA_3}(\balpha)$ and
  the statement follows from type $\TA$.

  If $n\geq 5$ is odd, let us remove the leaf with value $\alpha$. One
  gets, using Lemma \ref{saute_mouton} to compute the rightmost term,
  \begin{equation}
    N_{\TD_n}(\alpha)=q N_{\TA_1}(-1/\alpha) N_{\TA_{n-3}}+(q-1)N_{\TA_{n-1}}.
  \end{equation}
  According to results in type $\TA$, one therefore has to distinguish
  the case $\alpha=1$. One then compute using Prop. \ref{valeurA}.

  If $n$ is even, let us remove the leaf with value $\alpha$. One gets
  \begin{equation}
    N_{\TD_n}(\alpha,\beta)=q N_{\TA_1}(-\beta/\alpha) N_{\TA_{n-3}}(-1/\alpha)+(q-1)N_{\TA_{n-1}}(\beta).
  \end{equation}
  According to Prop. \ref{valeurA}, one has to distinguish according
  to three alternatives: $\beta=\alpha$ or not, $\alpha=(-1)^{n/2}$ or
  not, and $\beta=(-1)^{n/2}$ or not. One can also use the symmetry
  exchanging $\alpha$ and $\beta$. In each case, one can compute the
  result using Prop. \ref{valeurA}.
\end{proof}

\begin{remark}
  One may wonder, in type $\TD_n$ with $n$ odd, if the homotopy type
  of $X_{\TD_n}(\alpha)$ (for $\alpha \not= 1$) is that of a sphere.
\end{remark}

\section{Type $\TE$}

We will now consider the Dynkin diagrams of type $\TE$.

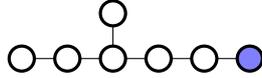
\begin{center}
\begin{tikzpicture}[scale=0.6]
\tikzstyle{every node}=[draw,shape=circle,very thick,fill=white]

\draw (0,0) node {} -- (1,0) node {} -- (2,0) -- (3,0) node {} -- (4,0) node {} -- (5,0) node[fill=blue!50]{};

\draw (2,1) node {} -- (2,0) node {};

\end{tikzpicture}
\end{center}

Using the general definition given for trees, one can introduce
varieties $X_{\TE_6}(\balpha)$, $X_{\TE_7}(\balpha)$ and
$X_{\TE_8}(\balpha)$ depending on invertible parameters $\balpha$.

\begin{proposition}
  Every variety $X_{\TE_6}(\balpha)$ is isomorphic to the variety
  $X_{\TE_6}(1,\dots,1)$. Every variety $X_{\TE_7}(\balpha)$ is
  isomorphic to the variety $X_{\TE_7}(1,\dots,1,\alpha)$, where
  $\alpha$ is the value on the last vertex on the long branch of
  $\TE_7$. Every variety $X_{\TE_8}(\balpha)$ is isomorphic to the
  variety $X_{\TE_8}(1,\dots,1)$.
\end{proposition}
\begin{proof}
  This follows from Proposition \ref{reduction}, using appropriate
  domino tilings.
\end{proof}

Let us introduce some notation for the number of points of these
varieties over finite fields. We will denote $N_{\TE_6}$ the number of points of
$X_{\TE_6}(1,\dots,1)$. and $N_{\TE_8}$ the number of points of
$X_{\TE_8}(1,\dots,1)$. We will denote $N_{\TE_7}(\alpha)$ the number
of points of $X_{\TE_7}(1,\dots,1,\alpha)$ where $\alpha$ is the value on the
last vertex on the long branch of $\TE_7$.

\begin{proposition}
  The number of points are as follows:
  \begin{equation}
    \begin{array}{rcl}
      N_{\TE_6} & = & q^6 + q^4 + q^3 + q^2 + 1,\\
      N_{\TE_7}(\alpha) & = & q^7 + q^5 - q^2 - 1\quad \text{ if }\quad\alpha\not=-1,\\
      N_{\TE_7}(-1) & = & q^7 + 2 q^5 +q^3 - q^2 - 1,\\
      N_{\TE_8} & = & q^8 + q^6 + q^5 + q^4 + q^3 + q^2 + 1.
    \end{array}
  \end{equation}
\end{proposition}

\begin{proof}
  The proof uses leaf-removal (Proposition \ref{leaf_removal}) and the
  knowledge of the numbers in type $\TA$.

  In the case of $\TE_6$, let us remove the leaf of the shortest
  branch. One gets, using Lemma \ref{saute_mouton} to compute the
  second term,
  \begin{equation}
    N_{\TE_6}=q N_{\TA_2} N_{\TA_2} + \sum_{\beta \in \Fq^*} N_{\TA_5}(\beta).
  \end{equation}
  One can compute this using Prop. \ref{valeurA} and Lemma \ref{valeurAsomme}. 
  
  In the case of $\TE_7$, let us remove the leaf of the shortest
  branch. One gets
  \begin{equation}
    N_{\TE_7}(\alpha) = q N_{\TA_2} N_{\TA_3}(-1/\alpha)+ (q-1) N_{\TA_6}.
  \end{equation}
  Therefore, by Prop. \ref{valeurA} applied to $\TA_3$, one has to
  separate the case $\alpha=-1$. One can compute the different results
  using Prop. \ref{valeurA}.
  
  In the case of $\TE_8$, let us remove the leaf of the shortest
  branch. One gets (using Lemma \ref{saute_mouton} to compute the
  second term)
  \begin{equation}
    N_{\TE_8}=q N_{\TA_2} N_{\TA_4} + \sum_{\beta \in \Fq^*} N_{\TA_7}(\beta).    \end{equation}
  One can compute this using Prop. \ref{valeurA} and Lemma \ref{valeurAsomme}.

\end{proof}

\bibliographystyle{alpha}
\bibliography{clusterfq}

\newcommand{\etalchar}[1]{$^{#1}$}
\begin{thebibliography}{BMR{\etalchar{+}}06}

\bibitem[BFZ05]{cluster3}
Arkady Berenstein, Sergey Fomin, and Andrei Zelevinsky.
\newblock Cluster algebras. {III}. {U}pper bounds and double {B}ruhat cells.
\newblock {\em Duke Math. J.}, 126(1):1--52, 2005.

\bibitem[BMR{\etalchar{+}}06]{bmrrt}
Aslak~Bakke Buan, Robert Marsh, Markus Reineke, Idun Reiten, and Gordana
  Todorov.
\newblock Tilting theory and cluster combinatorics.
\newblock {\em Adv. Math.}, 204(2):572--618, 2006.

\bibitem[FG03]{fg}
V.~V. Fock and A.~B. Goncharov.
\newblock Cluster ensembles, quantization and the dilogarithm, 2003.

\bibitem[FST08]{fst}
Sergey Fomin, Michael Shapiro, and Dylan Thurston.
\newblock Cluster algebras and triangulated surfaces. {I}. {C}luster complexes.
\newblock {\em Acta Math.}, 201(1):83--146, 2008.

\bibitem[FZ02]{cluster1}
Sergey Fomin and Andrei Zelevinsky.
\newblock Cluster algebras. {I}. {F}oundations.
\newblock {\em J. Amer. Math. Soc.}, 15(2):497--529 (electronic), 2002.

\bibitem[FZ03]{cluster2}
Sergey Fomin and Andrei Zelevinsky.
\newblock Cluster algebras. {II}. {F}inite type classification.
\newblock {\em Invent. Math.}, 154(1):63--121, 2003.

\bibitem[FZ07]{cluster4}
Sergey Fomin and Andrei Zelevinsky.
\newblock Cluster algebras. {IV}. {C}oefficients.
\newblock {\em Compos. Math.}, 143(1):112--164, 2007.

\end{thebibliography}

\end{document}